\documentclass[a4paper,12pt,leqno,twoside]{article}
\usepackage{multirow}

\title
{
Bases for the derivation modules of
two-dimensional multi-Coxeter arrangements
and universal derivations
}
\author{Atsushi WAKAMIKO}
\date{}

\usepackage{theorem}
\usepackage{amsmath}
\usepackage{amscd}
\usepackage{latexsym}
\usepackage{amssymb}
\usepackage{enumerate}
\usepackage{ascmac,fancybox}
\usepackage[dvips]{graphicx,color}

\theoremstyle{plain}

\newtheorem{theorem}{Theorem}[section]
\newtheorem{proposition}[theorem]{Proposition}
\newtheorem{lemma}[theorem]{Lemma}

\theorembodyfont{\rmfamily} 
\newtheorem{definition}[theorem]{Definition}
\newtheorem{example}[theorem]{Example}

\newtheorem{proof}{Proof.}

\makeatletter
\@addtoreset{equation}{section} 

\@addtoreset{figure}{section}   

\makeatother

\def\Z{\mathbb{Z}}  
\def\R{\mathbb{R}}  

\def\im{\mathop{\mathrm{im}}\nolimits}      

\font\b=cmr10 scaled \magstep4
\def\bigsymu#1{\smash{\lower1.7ex\hbox{\b #1}}}  

\def\Der{\mathrm{Der}}   
\def\A{\mathcal{A}}      
\def\bfk{\mathbf{k}}    

\def\eop{\square}                  

\begin{document}

\maketitle  

\thispagestyle{empty}

{\small
\begin{center}
\textbf{Abstract}
\end{center}

Let $\A$ be an irreducible Coxeter arrangement
and $\bfk$ be a multiplicity of $\A$.
We study the derivation module $D(\A, \bfk)$.
Any
two-dimensional irreducible Coxeter arrangement
with even number of lines 
is decomposed into two orbits under the action of 
the
Coxeter group.
In this paper, 
we will {explicitly} construct a
basis for $D(\A, \bfk)$
assuming $\bfk$ is constant 
on each orbit.  
Consequently we will determine the exponents of $(\A, \bfk)$ 
under this assumption.   
For this purpose we develop a theory
of universal derivations and introduce
a map to deal with our exceptional cases.

\medskip


\section{Introduction}
Let $V$ be an $\ell$-dimensional Euclidean space with inner product $I$.
Let $S$ denote the symmetric algebra of the dual space $V^{*}$ over
$\R$.
Denote the $S$-module of $\R$-linear derivations
of $S$ by $\Der_{S}$.  
Let $F$ be the field of quotients of $S$
and $\Der_{F} $ be the $F$-vector space of $\R$-linear derivations of $F$.   
Let $W \subseteq O(V,I)$ be a finite irreducible reflection group 
(a Coxeter group) and
$\A$ be the corresponding \textbf{Coxeter arrangement},
i.e., $\A$ is the set of all reflecting hyperplanes of $W$.
An arbitrary map $\mathbf{k} \colon \A \to \Z$ is 
called a \textbf{multiplicity} of
$\A$.
We say that
the pair $(\A,\textbf{k})$ is a \textbf{multi-Coxeter arrangement}.
The $S$-module $D(\A, \textbf{k})$,
defined in Section 2, of derivations associated with
$(\A, \bfk)$
was introduced by Ziegler \cite{Ziegler}
when $\im \bfk \subseteq \Z_{\geq 0} $ and in
\cite{Abe} \cite{Abe-Terao} for any multiplicity $\bfk$.
We say that $(\A, \bfk)$ is {\bf free} if $D(\A, \bfk)$ is 
a free $S$-module. 
The polynomial degrees (=pdeg) \cite{Orlik-Terao}  
of a homogeneous $S$-basis for $D(\A, \bfk)$ are called 
the {\bf exponents} of $(\A, \bfk)$.  
If $\bfk\equiv 1$, then $D(\A, \bfk)$ coincides with the $S$-module
$D(\A)$ of logarithmic derivations and $(\A, \bfk)$ is free (e.g., \cite{Saito 1}\cite{Orlik-Terao}).  
More in general, when $\bfk$ is
a constant function, 
$(\A, \bfk)$ is free and we can explicitly
construct a basis using basic invariants and a
primitive derivation as
in
\cite{Abe-Terao}\cite{Terao 2}.
In the case that $\bfk$ is not constant, however,
we do not know 
how we can
construct a basis for $D(\A, \bfk)$
even
when
$\ell=2$.
The
main result of this paper
gives an explicit construction of a basis for the module 
$D(\A, \bfk)$ 
when
$\ell=2$ and the multiplicity $\bfk$ is $W$-equivariant:
$\bfk(H) = \bfk(w H)$ for any $w\in W$ and
$H\in\A$.

The structure of this paper is as follows:
In Section 2, we 
define and discuss
the {\bf
universal derivations} which will be used in
the subsequent sections.  
Theorem \ref{Dinverseuniversal}
is the key result there.
In Sections 3
and 4, we assume that $\ell = 2$.
Then  
$W = I_{2} (h)$ is isomorphic to the
dihedral group of order $2h$.   When $h$ is odd, 
$\A$ itself is the unique $W$-orbit.
Thus $\bfk$ is constant and we can construct a 
basis {  (e.g., see
\cite{Terao 2}\cite{Abe-Yoshinaga}\cite{Abe}\cite{Abe-Terao}).   
}
So we may
assume that 
$h$ is even with $h\geq 4$.
In this case, 
we have the $W$-orbit decomposition:
$\A = \A_{1} \cup \A_{2}$.
Then both $\A_{1} $ and $\A_{2} $ are again 
irreducible arrangements if $h \geq 6$
(or equivalently
if
$W \neq B_{2}$).  
The corresponding irreducible
Coxeter groups $W_{1} $ and $W_{2} $ 
are both isomorphic to $I_{2} (\frac{h}{2})$.
For $a_{1}, a_{2} \in\Z$, let $(a_{1} , a_{2} )$ denote the multiplicity
$\bfk : \A \rightarrow \Z$ with 
$\bfk(H) = a_{1} \,\,(H\in\A_{1} )$ 
and
$\bfk(H) = a_{2} \,\,(H\in\A_{2} )$.
We classify the set $\{(a_{1} , a_{2} ) \mid a_{1} , a_{2} \in\Z\}$ 
into sixteen cases. 
The first fourteen 
{  cases}
are listed in 
Table \ref{Table1}.
\begin{table}[h]
\label{Table1}
\begin{center}
\begin{tabular}{|c|c|c|c|}
\hline
$(a_{1} , a_{2} )$ & $\zeta$  & $\theta_{1}, \theta_{2}$  
& basis for $D(\A, (a_{1} , a_{2} ))$\\
\hline
 $(4p+1, 4q+1)$ 
& $E^{(2p,2q)}$ 
& $E, I^{*}(dP_{2})$    
&\multirow{14}{*}{\hfill { 
$\nabla_{\theta_{1}} \zeta,
\nabla_{\theta_{2}} \zeta$}
\hfill}
\\
\cline{1-3}
 $(4p-1, 4q-1)$  
& $E^{(2p,2q)}$ 
& $D, I^{*}(dQ/Q)$
& \\
\cline{1-3}
{ 
$(4p-1,4q+1)$} 
& 
{ 
$E^{(2p,2q)}$ 
}
& 
{ 
$I^{*}(dQ_{1}/Q_{1}), E$
}
& \\
\cline{1-3}
{ 
$(4p+1, 4q-1)$ 
}
& 
{ 
$E^{(2p,2q)}$ 
}
& 
{ 
$I^{*}(dQ_{2}/Q_{2}), E$
}
& \\
\cline{1-3}
{ 
 $(4p+1, 4q)$ 
}
& 
{ 
$E^{(2p,2q)}$ 
}
& 
{ 
\multirow{2}{*}{$E, I^{*}(dQ_{2})$}
}
& \\
\cline{1-2}
{ 
$(4p+3, 4q+2)$ 
}
& 
{ 
$E^{(2p+1,2q+1)}$
}
&
& 
\\
\cline{1-3}
$(4p-1, 4q)$ 
& $E^{(2p,2q)}$ 
& \multirow{2}{*}{
{ 
$D_{1}, I^{*}(dQ_{1}/Q_{1})$}}    
&
\\
\cline{1-2}
$(4p+1, 4q+2)$ 
& $E^{(2p+1,2q+1)}$ 
&
&
\\
\cline{1-3}
$(4p, 4q+1)$ 
& $E^{(2p,2q)}$ 
& \multirow{2}{*}{{ 
$E, I^{*}(dQ_{1})$}}    
&
\\
\cline{1-2}
$(4p+2, 4q+3)$ 
& $E^{(2p+1,2q+1)}$ 
& 
&
\\
\cline{1-3}
{ 
$(4p,4q-1)$ 
}
& 
{ 
$E^{(2p,2q)}$ 
}
& 
{ 
\multirow{2}{*}{$D_{2}, I^{*}(dQ_{2}/Q_{2})$}    
}
&
\\
\cline{1-2}
{ 
$(4p+2, 4q+1)$ 
}
& 
{ 
$E^{(2p+1,2q+1)}$ 
}
& 
&
\\
\cline{1-3}
{ 
$(4p,4q)$ 
}
& { 
$E^{(2p,2q)}$ 
}
& \multirow{2}{*}{{ 
$\partial_{x_{1} }, \partial_{x_{2} } $}    
}&
\\
\cline{1-2}
{ 
$(4p+2,4q+2)$ 
}
& { 
$E^{(2p+1,2q+1)}$ 
}
& 
&
\\
\hline
\end{tabular}
\end{center}
\caption{Bases for $D(\A, (a_{1} , a_{2} ))$
 (ordinary cases) ($p \geq 0$ or $q \geq 0$)} 
\end{table}
}
We call the fourteen cases {\bf ordinary}.
The remaining two cases,
which are when either $(a_{1}, a_{2}) = (4p, 4q+2)$ or
$(4p+2, 4q)$,  
are called to be
{\bf exceptional} because our basis construction method
in the ordinary cases does not work for the exceptional
ones.  The exceptional cases are listed in Table 2. 
The derivations $\zeta = E^{(s, t)} $ are universal. 
We will explain how to read the two Tables in Sections 3
and 4.
Section 3 is devoted to the ordinary cases
where the main tool is the {\bf
Levi-Civita connection}
\[
\nabla : 
\Der_{F} \times
\Der_{F} \rightarrow
\Der_{F}
\]
with respect to $I$ together with {\bf
primitive derivations}
$D$ and $D_{i}$ corresponding to 
$W$ and $W_{i} \, (i=1 , 2)$
respectively.
The recipe here is Abe-Yoshinaga's
theory developed in \cite{Abe-Yoshinaga}
and \cite{Abe}. 
The main ingredient in Section 4
is 
the
maps
\begin{align*} 
&\Phi^{(1)}_{\zeta}    :
\Der_{S} 
\rightarrow 
D(\A, (4p+2, 4q)),
\\
&\Phi^{(2)}_{\zeta}    :
\Der_{S} 
\rightarrow 
D(\A, (4p, 4q+2)),
\end{align*} 
defined by
\[
\Phi_{\zeta}^{(1)}  (\theta) :=
Q_{1} (\nabla_{\theta} \,\zeta) -
(4p+1)   \theta(Q_{1}) \zeta,
\,\,
\Phi_{\zeta}^{(2)}  (\theta) :=
Q_{2} (\nabla_{\theta} \,\zeta) -
(4q+1)   \theta(Q_{2}) \zeta,
\]
where $Q_{i} $ is a defining polynomial for $\A_{i}
\,\,(i=1, 2)$
and $\zeta$ is $(2p, 2q)$-universal. 
\begin{table}[h]
\label{Table2}
\begin{center}
\begin{tabular}{|c|c|c|c|}
\hline
$(a_{1}, a_{2})$ & $\zeta$  & $\theta_{1}, \theta_{2}$  
& basis for $D(\A, (a_{1}, a_{2}))$\\
\hline
$(4p+2, 4q)$ 
& $E^{(2p,2q)}$ 
& $\partial_{x_{1} }, \partial_{x_{2} }  $ 
& $
\Phi^{(1)}_{\zeta} (\theta_{1}),
\Phi^{(1)}_{\zeta} (\theta_{2})
$ \\
\hline
$(4p, 4q+2)$ 
& $E^{(2p,2q)}$ 
& $\partial_{x_{1} }, \partial_{x_{2} }  $ 
& $
\Phi^{(2)}_{\zeta} (\theta_{1}),
\Phi^{(2)}_{\zeta} (\theta_{2})
$ \\
\hline
\end{tabular}
\end{center}
\caption{Bases for $D(\A, (a_{1}, a_{2}))$
(exceptional cases)
$(p \geq 0$ or $q \geq 0$)}   
\end{table}
Actually in Sections 3 and 4, we will construct bases
only
when either $p  \geq 0$ or $q \geq 0$
in Tables 1 and 2.
Lastly we cover 
the remaining cases
using the duality:
the existence of a non-degenerate $S$-bilinear pairing 
\[
\Omega(\A, \bfk) \times
D(\A, \bfk)
\longrightarrow
S,
\]
where $\Omega(\A, \bfk)
$ is the $S$-module of logarithmic differential $1$-forms
associated with the multi-Coxeter
arrangement
$(\A, \bfk)$
defined in \cite{Ziegler}, \cite{Abe} and
\cite{Abe-Terao2}.
We conclude this paper with Section 5 in which we 
present Table 4 showing the exponents
of $(\A, \bfk)$. 

\medskip

{ 
{\bf Remark}
In addition to $I_{2} (h)$ with $h\geq 4$ even,
there exist two kinds of irreducible Coxeter arrangements
which have two $W$-orbits: $B_{\ell}$ $(\ell \geq 2)$ and
$F_{4}$.    For each of these two cases,
when $\bfk$ is an equivariant multiplicity,
a basis for $D(\A, \bfk)$ is constructed
with a method similar to the one applied to
the ordinary cases here.   Details are found
in  \cite{Abe-Terao-Wakamiko}. 
}

\section{Universal derivations}

Let $\A$ be an irreducible Coxeter arrangement.
For each hyperplane $H \in \A$, choose a linear form
$
\alpha_H \in V^* 
$ 
such that $\ker(\alpha_H) = H$. 
The product $Q:= \prod_{H\in\A} \alpha_{H} $ 
lies in $S$. 
Let $\Omega_{S}  $ be the $S$-module of regular $1$-forms
and
$\Omega_{F} $ be the $F$-vector space of rational $1$-forms on $V$.
Let $I^{*}$ denote the inner product on $V^{*} $ induced from the
inner product $I$ on $V$. 
Then $I^{*} $ naturally
induces
an $S$-bilinear map $I^{*} : \Omega_{F} \times \Omega_{F} 
\rightarrow F$.   Thus we have an $F$-linear isomorphism
\[
I^{*} :
\Omega_{F} \rightarrow \Der_{F} 
\]
 by $[I^{*} (\omega)](f) = I^{*}(\omega, df) $ 
where $\omega\in\Omega_{F}, f\in F $. 
Recall the $S$-module
\begin{multline*}
\Omega(\A, \infty) := 
\{
\omega \in \Omega_{F} 
\mid
Q^{N} \omega
\mbox{\rm ~and~} 
(Q/\alpha_{H})^{N} \omega\wedge d\alpha_{H} \\
\mbox{\rm ~are both regular for any~} H\in\A \mbox{\rm ~and~} 
N \gg 0
\}
\end{multline*} 
of logarithmic {$1$-forms
\cite{Abe-Terao}.} 
We also have the $S$-module 
\begin{align*}
D(\A, -\infty) :&= I^{*} (\Omega(\A, \infty))\\
&=
\{
\theta \in \Der_{F} 
\mid
Q^{N} \theta\in \Der_{S} 
\mbox{\rm ~and~} 
(Q/\alpha_{H})^{N} \theta(\beta) 
\mbox{\rm ~is regular for~} \beta\in V^{*} \\
&\mbox{\rm ~~~~~~~~~~~~~~~~~~~~whenever~} I^{*} (\beta, \alpha_{H}) = 0
\mbox{\rm ~for any~}  
H\in\A \mbox{\rm ~and~} 
N \gg 0
\}
\end{align*} 
of logarithmic derivations \cite{Abe-Terao}.  
Let 
\begin{eqnarray*}
\nabla \colon \Der_{F} \times \Der_{F} &\longrightarrow& \Der_{F} \\
(\theta,\delta)\hspace{0.5cm} &\longmapsto& \nabla_\theta\, \delta
\end{eqnarray*}
be the
{Levi-Civita connection} with respect to $I$.
The derivation $\nabla_\theta \delta \in \Der_F$
 is characterized by the {equality} 
$(\nabla_\theta \delta)(\alpha) = \theta(\delta(\alpha))$ for any 
$\alpha \in
V^*$.

For $\alpha\in V^{*} $ let $S_{(\alpha)} $ denote
the localization of $S$ at the prime ideal $(\alpha)$ of $S$.
For an arbitrary multiplicity $\bfk : \A \rightarrow \Z$,    
define an $S$-submodule $D(\A,\bfk)$ of $D(\A, -\infty)$
 by 
\[
D(\A,\mathbf{k}) := 
\{\ 
\theta \in D(\A, -\infty) 
\mid \theta(\alpha_H) \in \alpha_{H}^{\mathbf{k}(H)}S_{(\alpha_{H})} \ 
\mbox{\rm for any $H \in \A$}\ \}
\]
from \cite{Abe-Terao2}. 
The module $D(\A,\textbf{k})$ was introduced by Ziegler \cite{Ziegler}
when $\im \bfk \subseteq \Z_{\geq 0}$.
Note $D(\A, {\bf 0}) = \Der_{S} $ 
where ${\bf 0}$ is the zero multiplicity.
For each $\mathbf{k} \colon \A \to \Z$,
define 
$
Q^{\mathbf{k}} 
:= \prod\nolimits_{H \in \A} \alpha_H^{\mathbf{k}(H)} 
\in F$.
Recall the following generalization of Saito's criterion
\cite{Saito 2}:

\begin{theorem}
\label{T. Abe's criterion 2 -theorem-}
\renewcommand{\theenumi}{\alph{enumi}}
\renewcommand{\labelenumi}{(\theenumi)}
\mbox{\rm (Abe \cite[Theorem 1.4]{Abe})}
Let $\mathbf{k} \colon \A \to \Z$ and 
$\theta_1,\ldots,\theta_\ell \in D(\A,\mathbf{k})$.
Then 
$\theta_1,\ldots,\theta_\ell$ form 
an $S$-basis for $D(\A,\mathbf{k})$
if and only if
$
\det[\theta_j(x_i)] \doteq
Q^{\mathbf{k}}.$
Here
$\doteq$ implies the equality up to a 
non-zero
constant multiple.
\end{theorem}

\begin{definition}
Let $\bfk : \A\rightarrow \Z$ and  
$\zeta \in D(\A, -\infty)^{W} $, where the superscript $W$ stands for the 
$W$-invariant part.
We say that $\zeta$ is {\bf
$\bfk$-universal}  when 
$\zeta$ is homogeneous and
the $S$-linear map
\begin{align*}
\Psi_{\zeta} : \Der&_{S} \longrightarrow D(\A, 2\bfk)\\
&\theta \longmapsto \nabla_{\theta}\, \zeta
\end{align*}
 is bijective.
\end{definition}

\begin{example}
The {\bf Euler derivation} $E$, which is the derivation characterized by
$E(\alpha) = \alpha$ for any $\alpha\in V^{*} $, is ${\bf 0}$-universal
because ${
\Psi_{E} (\delta)} = \nabla_{\delta}\, E = \delta$.    
\end{example}

For an irreducible Coxeter group $W$,
there exist algebraically independent homogeneous polynomials 
$P_{1} , P_{2} , \dots , P_{\ell} $ 
with
$\deg P_{1} < \deg P_{2} \leq \dots \leq P_{\ell-1} <\deg P_{\ell} $ 
by {Chevalley}'s Theorem \cite{Chevalley},
which are called {\bf basic invariants}.
When $D\in \Der_{F} $ {satisfies
}
\[
D(P_{j} ) =
\begin{cases}
0  \,\, {\rm if~}1\leq j < \ell,\\
1  \,\, {\rm if~}j = \ell,
\end{cases}  
\]
we say that $D$ is a {\bf primitive derivation.}
It is unique up to a nonzero constant multiple.
Let $R := S^{W} $ be the
$W$-invariant subring of $S$ and
$$T := \{
f\in R ~|~ D(f) = 0
\}.$$
\begin{theorem}
\label{Tautomorphism} 
\mbox{
\rm
(
\cite[Theorem 3.9 (1)]{Abe-Terao}
\cite[Theorem 4.4]{Abe-Terao2}
)}
(1)
We have a $T$-linear automorphism
\begin{eqnarray*}
\nabla_{D} : D(\A, -\infty)^{W} 
&\longrightarrow
&D(\A, -\infty)^{W},\\
\theta~~~~~
&\longmapsto
&~~~~\nabla_{D} \,\theta
\end{eqnarray*} 
(2)
$
\nabla_{D} (D(\A, 2 \bfk + {\bf 1} )^{W})
=
D(\A, 2\bfk - {\bf 1} )^{W}
$ 
for any multiplicity $\bfk : \A \rightarrow \Z$. 

\end{theorem}
Note that $\nabla_{D}^{-1}  $ and $\nabla_{D}^{k} \ (k\in\Z) $
are also $T$-linear automorphisms.

Let $x_{1}, \dots, x_{\ell}$
be a basis for $V^{*} $.
Put $A := [I^{*} (x_{i}, x_{j}) ]_{ij}$
which is a non-singular real symmetric matrix.
For simplicity
let 
$\partial_{x_{j} } $ 
and
$\partial_{P_{j} } $ 
denote 
$\partial/\partial x_{j} $ 
and
$\partial/\partial P_{j} $ 
respectively.
Note that $D = \partial_{P_{\ell} } $. 

\begin{proposition}
Let $k\in\Z$.  
Here $\bfk$ is a constant multiplicity: $\bfk \equiv k$. 
Then
the derivation $\nabla_{D}^{k}E  $ is 
$(-\bfk)$-universal.  
\end{proposition}

\begin{proof}
When $k\leq 0$, the result was first proved by Yoshinaga in
\cite{Yoshinaga}.
Assume $k > 0$.
Recall a basis
$\eta_{1}^{(-2k)},
\dots,
\eta_{\ell}^{(-2k)}
$ 
for $D(\A, -2k)$
introduced in \cite[Definition 3.1]{Abe-Terao}. 
Then we have
\[
[
\nabla_{\partial_{x_{1} }} \nabla_{D}^{k} E,
\dots,
\nabla_{\partial_{x_{\ell}}} \nabla_{D}^{k} E
]
 =
[
\eta_{1}^{(-2k)},
\dots,
\eta_{\ell}^{(-2k)}
]
A^{-1}, 
\]
which is the second equality of
\cite[Proposition 4.3]{Abe-Terao}
(in the differential-form version).
$\square$ 
\end{proof}

\begin{proposition}
\label{123} 
Let $\zeta\in D(\A, -\infty)^{W} $
be $\bfk$-universal.  Then

(1)
the $S$-linear map
\begin{eqnarray*}
\Psi_\zeta \colon 
D(\A, {\bf -1}) &\longrightarrow& D(\A, 2\bfk{\bf -1})\\
\theta \hspace{5mm} &\longmapsto& \hspace{5mm} \nabla_\theta \, \zeta
\end{eqnarray*}
is bijective,

(2)
$\zeta\in D(\A, 2\bfk + {\bf 1})^{W} $,
and

(3)
$\alpha_{H}^{-2\bfk(H)-1} \zeta(\alpha_{H} ) $ is a unit in 
$S_{(\alpha_{H} )} $ for any $H\in\A$.     
\end{proposition}

\begin{proof} 
{(1)}
Note
that $\partial_{P_{1} }, \dots, \partial_{P_{\ell} }$ 
form an $S$-basis for $D(\A, {\bf -1})$
\cite[p.823]{Abe-Terao}.
Let $1\leq j\leq \ell$.
Then
\[
Q\nabla_{\partial_{P_{j} } }  \zeta
=
\nabla_{Q\partial_{P_{j} } }  \zeta
\in 
D(\A, 2\bfk)
\]
because $Q\partial_{P_{j} }\in\Der_{S} $.
Thus 
$$
\left(\nabla_{\partial_{P_{j} }  } \zeta\right)
(\alpha_{H} )\in \alpha_{H}^{2\bfk(H) - {1}} S_{(\alpha_{H})} 
\,\,\,\,\,\,(H\in\A). 
$$ 
Pick $H\in\A$ arbitrarily and choose an orthonormal basis
$x_{1} , \dots, x_{\ell} $ 
for $V^{*} $ so that $H=\ker(x_{1} )$.
For $i = 2, \dots, \ell$ define
$
g_{i} := {
(Q/x_{1})^{N}} 
Q (\nabla_{\partial_{P_{j} }} {
\zeta} )(x_{i} )\in S
$
{
for a sufficiently large positive integer $N$. }
 Let $s = s_{H} $ denote the orthogonal reflection through 
$H$.  Then
$s(g_{i} )= - g_{i} $.
Thus $g_{i} \in x_{1} S$   
and
\[
(\nabla_{\partial_{P_{j} } } \zeta)(x_{i} ) = 
{
(Q/x_{1})^{-N}}
 g_{i} /Q \in {
S_{(x_{1})}}.
\]
This implies $\nabla_{\partial_{P_{j} } }\zeta \in D(\A, -\infty)$ 
and thus $\nabla_{\partial_{P_{j} } }  \zeta\in
D(\A, 2\bfk-{\bf 1}).
$  
One has
\begin{align*}
\det\left[
\left(
\nabla_{\partial_{P_{j} } }  \zeta
\right)
(x_{i} )
\right]
&=
\det
\left(
\left[
\left(
\nabla_{\partial_{x_{j} } } \zeta
\right)
(x_{i} )
\right]
\left[
\partial{P_{i} }
/
\partial{x_{j} }
\right]^{-1} 
\right)
\doteq
Q^{-1} \det
\left[
\left(
\nabla_{\partial_{x_{j} } } \zeta
\right)
(x_{i} )
\right]\\
&\doteq
Q^{2\bfk - {\bf 1}} 
\end{align*} 
by the chain rule
$
\partial_{x_{j} }
=
\sum_{s=1}^{\ell}   
\left(\partial{P_{s} }
/
\partial{x_{j} }
\right)
\partial_{P_{s} }$
and
the equality
$\det
\left[
\partial{P_{i} }
/
\partial{x_{j} }
\right]
\doteq
Q.
$ 
Applying Theorem \ref{T. Abe's criterion 2 -theorem-} 
we conclude that
$
\nabla_{\partial_{P_{1} } } \zeta,
\dots,
\nabla_{\partial_{P_{\ell} } } \zeta
$ 
form an $S$-basis for $D(\A, 2\bfk - {\bf 1})$.

{(2)}
By {(1)}, $\nabla_{D}\, \zeta\in D(\A, 2\bfk - {\bf 1})^{W}$.
Thanks to Theorem \ref{Tautomorphism}, we have $\zeta\in D(\A, 2\bfk + {\bf 1})^{W}$.

{(3)}
By {(2)}, 
$\zeta(\alpha_{H} ) 
\in
\alpha_{H}^{2\bfk(H)+1} 
S_{(\alpha_{H} )} $ for any $H\in\A$.     
Assume that
$\alpha_{H}^{-2\bfk(H)-1} \zeta(\alpha_{H} ) 
$
is not a unit in $S_{(\alpha_{H})}$
for some $H\in\A$.
Choose an orthonormal basis
$x_{1}, x_{2}, \dots, x_{\ell}   $ 
for $V^{*} $ so that $H
= \ker (x_{1} )$.  
Then
$
\zeta(x_{1} ) 
\in
x_{1}^{2\bfk(H)+2} 
S_{(x_{1} )}
$.
Thus
$
(\nabla_{\partial_{x_{j} } } 
\zeta)(x_{1} ) 
\in
x_{1}^{2\bfk(H)+1} 
S_{(x_{1} )}
$ for each $j$ with $1\leq j\leq \ell$ 
and 
$
Q^{2 \bfk} 
\doteq
\det 
\left[
(\nabla_{{
\partial_{x_{j} }}} 
\zeta)(x_{i} )
\right]
\in
x_{1}^{2\bfk(H)+1} 
S_{(x_{1} )},
$ 
which is a contradiction.
$\square$ 
\end{proof}

\begin{proposition}
\label{psiuniversal}
\mbox{\rm (cf. \cite[Theorem 10]{Abe-Yoshinaga}\cite[Theorem 2.1]{Abe})}
If $\zeta\in D(\A, -\infty)^{W}$ is $\bfk$-universal
and
 ${\bf m} : \A \rightarrow \{-1, 0, 1\}$ is a multiplicity,
then the $S$-linear map
\begin{eqnarray*}
\Psi_\zeta \colon 
D(\A, {\bf m}) &\longrightarrow& D(\A, 2\bfk+{\bf m})\\
\theta \hspace{3mm} &\longmapsto& \hspace{4.5mm} \nabla_\theta \,\zeta
\end{eqnarray*}
is bijective.   
\end{proposition}

\begin{proof}
Note that 
$D(\A, {\bf m})
\subseteq D(\A, -{\bf 1})$
and
$D(\A, \bfk+{\bf m})
\subseteq D(\A, \bfk-{\bf 1})$.
By Proposition \ref{123} {(1)},
the restriction of 
\begin{eqnarray*}
\Psi_\zeta \colon 
D(\A, {\bf -1}) &\longrightarrow& D(\A, 2\bfk{\bf -1})
\end{eqnarray*}
to $D(\A, {\bf m})$ is injective.
Thus it is enough to prove
$\Psi_\zeta(D(\A, {\bf m})) = D(\A, 2\bfk+{\bf m}).$
Let 
$\theta\in D(\A, -{\bf 1})$.
Pick $H\in\A$ arbitrarily and fix it.
Choose an orthonormal basis
$x_{1} , x_{2} , \dots , x_{\ell} $ 
with $H = \ker (x_{1} )$.
Let 
$k := {\bf k} (H)$ 
and 
$m := {\bf m} (H)$.
Then, by Proposition \ref{123} (3),
$g:= x_{1}^{-2k-1}  \zeta(x_{1} )$ is a unit in $S_{(x_{1} )}$.   
Compute
\begin{align*} 
(\Psi_{\zeta} (\theta))(x_{1})
&=
(\nabla_{\theta} \, \zeta)(x_{1})
=
\theta(\zeta(x_{1}))
=
\theta(x_{1}^{2k+1} g)
=
x_{1}^{2k+1} \theta(g) +
(2k+1) x_{1}^{2k} \theta(x_{1}) g\\
&=
x_{1}^{2k+1} 
\sum_{j=1}^{\ell}  \theta(x_{j})(\partial g/\partial x_{j}) 
+
(2k+1) x_{1}^{2k} \theta(x_{1}) g\\
&=
x_{1}^{2k} \theta(x_{1})
\left\{
x_{1} (\partial g/\partial x_{1})
+(2k+1)g
\right\}
+
x_{1}^{2k+1} 
\sum_{j=2}^{\ell}  \theta(x_{j})(\partial g/\partial x_{j}) \\
&
=
x_{1}^{2k} \theta(x_{1})
U+
x_{1}^{2k+1} 
C,
\end{align*} 
where $U := x_{1} (\partial g/\partial x_{1})
+(2k+1)g
$ 
is a unit in $S_{(x_{1} )} $
and $C :=
\sum_{j=2}^{\ell}  \theta(x_{j})(\partial g/\partial x_{j}).
$
Dividing the both sides by $x_{1}^{2k+m}$, we get
\[
x_{1}^{-2k-m}  
(\Psi_{\zeta} (\theta))(x_{1})
=
x_{1}^{-m} \theta(x_{1})
U+
x_{1}^{1-m} 
C.
\]
Note that $  \partial g/\partial x_{j}
\in
S_{(x_{1} )}$
and
$ 
\theta(x_{j})  \in 
S_{(x_{1} )} 
\ (j\geq 2)$ because $\theta\in D(\A, -\infty)$.
So one has 
$
C \in S_{(x_{1} )} $
and
$x_{1}^{1-m} 
C \in S_{(x_{1} )}$
for $m\in\{\pm 1, 0\}$.
Thus we conclude that 
$$
x_{1}^{-2k-m} (\Psi_{\zeta} (\theta))(x_{1}) 
\in S_{(x_{1} )} 
\Longleftrightarrow
x_{1}^{-m} \theta(x_{1}) 
\in S_{(x_{1} )}. $$ 
This implies that
$$
\Psi_{\zeta} (\theta)
\in
D(\A, 2{\bf k}+{\bf m})
\Longleftrightarrow
\theta
\in 
D(\A, {\bf m})$$
because $H\in\A$ was arbitrarily chosen. 
This completes the proof.
$\square$ 
\end{proof}

The following is the main result in this section.

\begin{theorem}\label{Dinverseuniversal}
Let $\mathbf{k} \colon \A \to \Z$ be a multiplicity of $\A$.
Let $\zeta \in D(\A,-\infty)^{W}$ be $\bfk$-universal.
Then
$
\nabla_D^{-1} \zeta
$
is $(\bfk+{\bf 1})$-universal.
\end{theorem}
\begin{proof}
It is classically known 
\cite{Saito 1}
that 
$\xi_{j} := I^{*} (dP_{j} ) \in D(\A, {\bf 1} )^{W} \
(j = 1,\dots, \ell)$
form an $S$-basis for $D(\A, {\bf 1})$.  
By Proposition \ref{psiuniversal}, 
$\nabla_{\xi_{j}} \zeta \in D(\A, 2\bfk+{\bf 1} )^{W} \
(j = 1,\dots, \ell)$
form an $S$-basis for $D(\A, 2\bfk+{\bf 1})$.  
Since 
$\nabla_{D} \nabla_{\xi_{j}} \zeta \in D(\A, 2\bfk-{\bf 1} )^{W} \
(j = 1,\dots, \ell)$
by Theorem \ref{Tautomorphism},
we can write
\[
\nabla_D \nabla_{\xi_j} \zeta
= \sum\nolimits_{i=1}^\ell f_{ij} \nabla_{\partial_{P_i}}\zeta
\]
with $W$-invariant polynomials $f_{ij} \in R$
because of Proposition \ref{123} (1).
Then $f_{ij}$ is a homogeneous element with degree $m_i+m_j-h < h$,
where $h$ is the Coxeter number,
and  $f_{ij}$ belongs to 
$
T =\{f \in R \mid Df=0\}.
$
Since $m_{i} + m_{\ell+1-i} - h =0$,
$\det [f_{ij} ]\in\R$.  
Apply $\nabla_{D}^{-1} $ to the both sides
to get
\[
\nabla_{\xi_j} \zeta
= \nabla_{D}^{-1}  \sum\nolimits_{i=1}^\ell f_{ij} 
\nabla_{\partial_{P_i}}
\zeta
=\sum\nolimits_{i=1}^\ell f_{ij} \nabla_{\partial_{P_i}}
  \nabla_{D}^{-1} 
 \zeta.
\]
Since
$\nabla_{\xi_{j}} \zeta \in D(\A, 2\bfk+{\bf 1} )^{W} 
\
(j = 1,\dots, \ell)$
form an $S$-basis for $D(\A, 2\bfk+{\bf 1})$,
we have 
$\det [f_{ij} ]\in\R^{\times}$.  
This implies that
$
\nabla_{\partial_{P_j}}
  \nabla_{D}^{-1} 
\zeta
\
(j = 1,\dots, \ell)$
form an $S$-basis for $D(\A, 2\bfk+{\bf 1})$.
Since
$
  \nabla_{D}^{-1} 
\zeta
\in 
D(\A, 2\bfk+{\bf 3} )
$
by
Proposition \ref{123} (2)
and Theorem \ref{Tautomorphism},
we conclude that  
$$
\nabla_{\partial_{x_j}}
  \nabla_{D}^{-1} 
 \zeta 
=
\sum_{i=1}^{\ell} 
(\partial_{x_{j} } P_{i})  
\nabla_{\partial_{P_i}}
  \nabla_{D}^{-1} 
\zeta
\,\,\,
(j = 1,\dots, \ell)$$
form an $S$-basis for $D(\A, 2\bfk+{\bf 2})$
by Theorem \ref{T. Abe's criterion 2 -theorem-}.
$\eop$
\end{proof}

\section{The ordinary cases}
In the rest of this paper
 we assume $\dim V = \ell =2$
and $W = I_{2}(h)$
such that $h\geq 4$ is an even number.
The orbit decomposition 
$\A = \A_{1} \cup \A_{2}$ satisfies
$|\A_{1} |= |\A_{2} |= h/2$.
Recall the equivariant multiplicities
$\bfk = (a_{1} , a_{2} )$,
$a_{1} , a_{2} \in\Z$,  defined by 
\[
\bfk (H)
=
\begin{cases}
a_{1}   \, \, {\rm ~if~} H \in \A_{1},\\ 
a_{2}   \, \, {\rm ~if~} H \in \A_{2}.\\ 
\end{cases}  
\]

Let $x_{1}, x_{2}$ be an orthonormal basis for $V^{*} $.
Suppose that $P_{1} := (x_{1}^{2} + x_{2}^{2})/2$
and $P_{2} $ are basic invariants of $W$. 
 Then
$\deg P_{2} = h$
and
$R = S^{W} = \R[P_{1}, P_{2}]$. 
Let $W_{i} $ be the {
(normal)} subgroup of $W$ generated by
all reflections through $H\in\A_{i} \,\,(i=1,2)$. 
Let $Q_{i} = \prod_{H\in \A_{i} } \alpha_{H} $
and
$R_{i} := S^{W_{i} }\,\,(i=1,2 )$.
Let
$D$ be a primitive derivation corresponding to
the whole group $W$.
Then it is known \cite[(5.1)]{Solomon-Terao} 
that
\[
D \doteq 
\frac{1}{Q} 
\left(
-x_{2} \partial_{x_{1} }  
+
x_{1} \partial_{x_{2} }
\right).  
\]

\begin{lemma}
\label{lemma3.1} 
Define 
$$
D_{1} := Q_{2} D
\doteq
\frac{1}{Q_{1} } 
\left(
-x_{2} \partial_{x_{1} }  
+
x_{1} \partial_{x_{2} }
\right),
\,\,\,\,
D_{2} := Q_{1} D
\doteq
\frac{1}{Q_{2} } 
\left(
-x_{2} \partial_{x_{1} }  
+
x_{1} \partial_{x_{2} }
\right).
$$
Then

(1) 
$R_{1} =\R[P_{1}, Q_{2}]$,
$R_{2} =\R[P_{1}, Q_{1}]$
 and
$R =
\R[P_{1}, Q_{1}^{2}]
=
\R[P_{1}, Q_{2}^{2}]
$,

(2)
$
-x_{2} (\partial Q_{2}/\partial x_{1})
+
x_{1} (\partial Q_{2}/\partial x_{2})
\doteq
Q_{1} 
$ 
and
$
-x_{2} (\partial Q_{1}/\partial x_{1})
+
x_{1} (\partial Q_{1}/\partial x_{2})
\doteq
Q_{2},
$ 

(3) $D_{1} (P_{1} ) = D_{2} (P_{1} ) = 0$,
$D_{1} (Q_{2} ) \in \R^{\times} $
and
$D_{2} (Q_{1} ) \in \R^{\times} $.
\end{lemma}

\begin{proof} 
Thanks to the symmetry we only have to prove a half of the statement.
Since $Q$ and $Q_{1} $ are both $W_{1} $-antiinvariant,
$Q_{2} = Q/Q_{1} $ is $W_{1} $-invariant
and $Q_{2}^{2}  $ is $W$-invariant.
Note that $Q_{2} $ is a product of real linear forms.
So $Q_{2} $ and $P_{1} $ are algebraically independent.
Since 
\[
|\A_{1} |
=
h/2
=
(\deg Q_{2} - 1)
+
(\deg P_{1} - 1),
\]
we have 
$R_{1} = \R[P_{1}, Q_{2}]$.
Similarly we obtain
$R = \R[P_{1}, Q_{2}^{2} ]$.
 This proves (1).
The Jacobian
\[
-x_{2} (\partial Q_{2}/\partial x_{1})
+
x_{1} (\partial Q_{2}/\partial x_{2})
=
\det
\begin{pmatrix}
\partial P_{1} /\partial x_{1} 
&
\partial Q_{2} /\partial x_{1} \\
\partial P_{1} /\partial x_{2} 
&
\partial Q_{2} /\partial x_{2} 
\end{pmatrix}
\neq 0  
\]
is equal to $Q_{1} $ 
up to a nonzero constant multiple, which is (2).
Compute
\[
D_{1} (P_{1} ) = Q_{2} D(P_{1} ) = 0,
\,\,
2 D_{1} (Q_{2} ) = 2 Q_{2} D(Q_{2} ) = D(Q_{2}^{2}  )\in\R^{\times}.
\]
This proves (3).
$\square$ 
\end{proof} 

The {Euler derivation} 
$E=
I^{*} (dP_{1} )
=
I^{*} (x_{1} dx_{1} + x_{2} dx_{2})
=
x_{1} \partial_{x_{1} } 
+
x_{2} \partial_{x_{2} } 
$ 
satisfies $E(\alpha) = \alpha$ for all $\alpha\in V^{*} $
and belongs to $D(\A, (1,1))$.

\begin{proposition}
\label{proposition3.2} 
A basis for $D(\A, (a_{1} ,a_{2} ))$
is given in Table 3 for 
$
-1\leq a_{1} \leq 1,
-1\leq a_{2} \leq 1.
$  
\end{proposition}

\begin{table}[h]
\label{Table3}
\begin{center}
\begin{tabular}{|c|c|c|c|}
\hline
  $(a_{1}, a_{2})$ 
&basis for $D(\A, (a_{1} , a_{2} ))$
&exponents of  $(\A, (a_{1} , a_{2} ))$
&their difference
\\
\hline
  $(1, 1)$ 
& $E, I^{*}(dP_{2})$
& $ 1, h-1$
&
$h-2$
\\
\hline
  $(1, 0)$ 
& $E, I^{*}(dQ_{2})$
& $ 1, (h/2)-1$
&
$(h/2)-2$
\\
\hline
  $(0, 1)$ 
& $E, I^{*}(dQ_{1})$
&
$ 1, (h/2)-1$
&
$(h/2)-2$
\\
\hline
  $(1, -1)$ 
& $ I^{*}(dQ_{2}/Q_{2} ), E$
& $ -1, 1$
&
$2$
\\    
\hline
  $(0, 0)$ 
& $\partial_{x_{1} }, \partial_{x_{2} } $
& $ 0, 0$
&
$0$ \\    
\hline
  $(-1, 1)$ 
& $ I^{*}(dQ_{1}/Q_{1} ), E$    
& $ -1, 1$
&
$2$
\\    
\hline
  $(0, -1)$ 
& $D_{2}, I^{*}(dQ_{2}/Q_{2} )$
& $ 1-(h/2), -1$
&
$(h/2)-2$
\\    
\hline
  $(-1, 0)$ 
& $D_{1}, I^{*}(dQ_{1}/Q_{1} )$
& $ 1-(h/2), -1$
&
$(h/2)-2$
\\    
\hline
  $(-1, -1)$ 
& $D, I^{*}(dQ/Q)$
& $ 1-h, -1$
&
$h-2$
\\    
\hline
\end{tabular}
\end{center}
\caption{The exponents of $(\A, (a_{1}, a_{2}))$ 
($
-1\leq a_{1} \leq 1,
-1\leq a_{2} \leq 1
$)}
\end{table}

\begin{proof}
Let $\omega_{0} =
-x_{2} d{x_{1} } 
+
x_{1} d{x_{2} }. 
$ 
Note that $\omega_{0} \wedge d\alpha = 
{
-\alpha}(dx_{1} \wedge dx_{2} )$ 
for any $\alpha\in V^{*} $.
It is easy to see that each of
$
dP_{1} , dP_{2} , dQ_{1} , dQ_{2} , dQ_{1}/Q_{1}  , dQ_{2}/Q_{2}  ,
\omega_{0}/Q,$
 $\omega_{0}/Q_{1}$
and $\omega_{0}/Q_{2}  
$
belongs to $\Omega(\A, \infty)$ defined in Section 1.
Note that 
$D = I^{*} (\omega_{0} )/Q$
and
$D_{i}  = I^{*} (\omega_{0} )/Q_{i} 
\,\,\,(i = 1, 2)$.
Thus all of the derivations in the table lie in 
$D(\A, -\infty) = I^{*}(\Omega(\A, \infty)) $.  
 
If $P$ is $W$-invariant, then  
$I^{*} (dP)\in D(\A, (1, 1)). $ 
Therefore
$I^{*} (dQ_{1} )\in D(\A, (0, 1)) $ 
and
$I^{*} (dQ_{2} )\in D(\A, (1, 0)) $
because of Lemma \ref{lemma3.1} (1). 
{
We thus have
$I^{*} (dQ_{1}/Q_{1})\in  D(\A, (-1,1))$
and
$
I^{*} (dQ_{2}/Q_{2})\in  D(\A, (1,-1)).
$
}
Since $QD = Q_{1} D_{1} 
= Q_{2} D_{2} $ lies in $\Der_{S} $,
we get 
$D\in D(\A, (-1,-1))$,
$D_{1} \in D(\A, (-1, 0))$
and
$D_{2} \in D(\A, (0, -1))$.
Now apply Theorem \ref{T. Abe's criterion 2 -theorem-}
noting Lemma \ref{lemma3.1} (2). 
$\square$ 
\end{proof}

\begin{lemma}
\label{lemma3.3} 
When $h\geq 6$ is even, $D_{i}$ is a primitive derivation of
the irreducible Coxeter arrangement $\A_{i} \,\,\,
(i=1, 2)$. 
\end{lemma}

\begin{proof}
By Lemma \ref{lemma3.1} (3).
$\square$  
\end{proof}

{For $s, t\in\Z$ with $t-s\in 2\Z$,} define
$$
E_{1}^{(s, t)} 
:=
\nabla_{D}^{-t} \nabla_{D_{1}}^{t-s} E,
\,\,\,\,\,\,\,\,\,
E_{2}^{(s, t)} 
:=
\nabla_{D}^{-s} \nabla_{D_{2}}^{s-t} E.
$$

\begin{proposition}
\label{Estuniversal} 
(1)
If $t \in \Z_{\geq 0}$ 
and $t-s \in 2\Z$, then 
$E_{1}^{(s, t)} $ is $(s, t)$-universal,

(2)
If $s \in \Z_{\geq 0}$ 
and $s-t \in 2\Z$, then 
$E_{2}^{(s, t)} $ is $(s, t)$-universal.
\end{proposition}

\begin{proof}
It is enough to show (1)
because of the symmetry of the statement.

{\it Case 1.} 
When $h\geq 6$ is even, 
$\A_{1} $ is an irreducible Coxeter arrangement of $h/2$ lines.  
By Lemma \ref{lemma3.3}, $D_{1}$ is a primitive derivation 
of $\A_{1} $. 
Thus $$
\nabla_{\partial_{x_{1} }} \nabla_{D_{1}}^{t-s} E,
\dots,
\nabla_{\partial_{x_{\ell} }} \nabla_{D_{1}}^{t-s} E
$$ 
form an $S$-basis for $D(\A, (2(s-t), 0))$.  
Note that
$D_{1}= Q_{2} D$ satisfies 
$$
w_{1} D_{1} = D_{1},
w_{2} D_{1} = \det(w_{2} ) D_{1}
$$
for any $w_{1} \in W_{1}, w_{2} \in W_{2} $. 
%
{Since $W_{1} $ is a normal subgroup of $W$,
$D(\A_{1}, -\infty)^{W_{1} } $ is naturally a
$W$-module and the map
$\nabla_{D_{1}}^{n} : D(\A_{1}, -\infty)^{W_{1} } \rightarrow
D(\A_{1}, -\infty)^{W_{1} } $ is a
$W$-equivariant bijection when $n$ is even.  
}
Thus 
$\nabla_{D_{1} }^{t-s} E \in D(\A, -\infty)^{W}$.
This implies 
that
$\nabla_{D_{1} }^{t-s} E$ is $(s-t, 0)$-universal
when ${t-s} \in 2\Z$.  
Apply Theorem \ref{Dinverseuniversal}.

{\it Case 2.} 
Let $h=4$.  Then $W$ is of type $B_{2} $.
We may choose an orthonormal basis for $V^{*} $ with
$Q_{1} = x_{1} x_{2} $ and
$Q_{2} = (x_{1} + x_{2})(x_{1} - x_{2}) $.
Then
\[
D_{1} =
-\frac{1}{x_{1} } \partial_{x_{1} } 
+
\frac{1}{x_{2} } \partial_{x_{2} }
\]
and
\begin{align*} 
&\nabla_{D_{1}}^{2n} E
=
-
{(4n-3)!!}
\left(
x_{1}^{1-4n} \partial_{x_{1} } 
+
x_{2}^{1-4n}  \partial_{x_{2} }
\right)
\in D(\A, -\infty)^{W}
\,\,\,\,(n > 0),
\\
&
\nabla_{D_{1}}^{-2n} E
=
\frac{1}{(4n+1)!!}
\left(
x_{1}^{4n+1} \partial_{x_{1} } 
+
x_{2}^{4n+1}  \partial_{x_{2} }
\right)
\in D(\A, -\infty)^{W}\,\,\,\,(n \geq 0),
\end{align*} 
{where $(2m-1)!! = \prod_{i=1}^{m} (2i-1) $.}
Thus
\[
\nabla_{\partial_{x_{1} } } 
\nabla_{D_{1}}^{2n} E \doteq
x_{1}^{-4n}\partial_{x_{1} }
,
\,\,\,
\nabla_{\partial_{x_{2} } } 
\nabla_{D_{1}}^{2n} E \doteq
x_{2}^{-4n}\partial_{x_{2} }
\,\,\,(n\in\Z).
\]
This implies 
that
$\nabla_{D_{1} }^{t-s} E$ is $(s-t, 0)$-universal
when $s-t \in 2\Z$.  
Apply Theorem \ref{Dinverseuniversal}.  
$\square$ 
\end{proof}

We say that a pair $(a_{1}, a_{2})$ is {\bf
exceptional}
if
\[
a_{1}\in 2\Z \mbox{~and~} 
a_{1} -a_{2} \equiv 2 \  ({\rm mod} \,4). 
\]
If $(a_{1}, a_{2})$ is not exceptional,
then we call $(a_{1}, a_{2})$ {
{\bf ordinary}.
We may apply 
Theorem \ref{proposition3.2} and Proposition \ref{psiuniversal} to get
the following proposition:

\begin{proposition}
\label{ordinarybasis} 
Suppose that $(a_{1}, a_{2})$ is ordinary
and that either $p\geq 0$ or
$q\geq 0$ in Table 1.
Then
$
\nabla_{{\theta_{1}}}\zeta,
\nabla_{{\theta_{2}}}\zeta
$ 
form an $S$-basis for $D(\A, (a_{1} , a_{2} ))$
as in Table 1, where 
$
E^{(s, t)}
$
stands for 
$
E_{1}^{(s, t)}$ if $t \geq 0$
or it
stands for 
$E_{2}^{(s, t)}$ if $s \geq 0.$
\end{proposition}

%
%

\section{The exceptional cases}
Suppose that $(a_{1}, a_{2}) \in \Z^{2} $ is exceptional.
Write
\[
(a_{1}, a_{2}) = (4p+2, 4q) \ \ \mbox{\rm or} \ \ 
(a_{1}, a_{2}) = (4p, 4q+2)
\,\,\,\,\,\,(p, q\in\Z).
\]

\begin{proposition}
\label{Phiisom} 
Suppose that $\zeta$ is $(2p, 2q)$-universal.
Then 
the map 
\begin{align*} 
\Phi^{(1)}_{\zeta} 
:~
&\Der_{S} 
\longrightarrow
D(\A, (4p+2, 4q))\\
&~~~~\theta
\longmapsto
Q_{1} (\nabla_{\theta} \zeta) 
-
(4p+1) \theta(Q_{1}) \zeta
\end{align*} 
is an $S$-linear {bijection.}
Similarly  
the map 
\begin{align*} 
\Phi^{(2)}_{\zeta} 
:~
&\Der_{S} 
\longrightarrow
D(\A, (4p, 4q+2)) \\
&~~~~\theta
\longmapsto
Q_{2} (\nabla_{\theta} \zeta) 
-
(4q+1) \theta(Q_{2}) \zeta
\end{align*} 
is an $S$-linear {bijection.}
\end{proposition}

\begin{proof}
It is enough to show the first half
because of the symmetry.
Let $\theta \in \Der_S$.  We
first prove that $\Phi^{(1)}_{\zeta} (\theta) 
\in D(\A, {(4p+2, 4q)
})$.
Let 
$H_{i}  \in \A_{i}$
and $\alpha_{i} := \alpha_{H_{i} } 
(i = 1, 2)$.
 Since $\zeta\in D(\A, (4p+1, 4q+1))$ 
by Proposition \ref{123} (2), write
\[
\zeta(\alpha_{1} ) = \alpha_{1}^{4p+1} f_{1},
\,\,\,
\zeta(\alpha_{2} ) = \alpha_{2}^{4q+1} f_{2}
\,\,\,\,\,\, {(f_{1} \in S_{(\alpha_{1})},  f_{2} \in S_{(\alpha_{2})})}.   
\]
Compute
\begin{align*}
[
\Phi^{(1)}_{\zeta}
(\theta)]
 (\alpha_{1}) 
&=
Q_{1} (\nabla_{\theta} \zeta)(\alpha_{1})
-(4p+1)\theta(Q_{1}) \zeta(\alpha_{1})\\
&=
Q_{1} ({\theta} (\alpha_{1}^{4p+1}f_{1}))-
(4p+1)\theta(Q_{1}) 
\alpha_{1}^{4p+1}f_{1}\\
&=
Q_{1} \alpha_{1}^{4p+1}{\theta} (f_{1})
+
(4p+1) f_{1} \alpha_{1}^{4p}  Q_{1} \theta(\alpha_{1}) 
- 
(4p+1)  f_{1} \alpha_{1}^{4p+1} \theta(Q_{1})\\
&=
Q_{1} \alpha_{1}^{4p+1}{\theta} (f_{1})
-
(4p+1) f_{1} \alpha_{1}^{4p+2}
\left\{
(1/\alpha_{1}) \theta(Q_{1}) 
- 
(Q_{1}/\alpha_{1}^{2})   \theta(\alpha_{1})
\right\}
\\
&=
Q_{1} \alpha_{1}^{4p+1}{\theta} (f_{1})
-
(4p+1) f_{1} \alpha_{1}^{4p+2}
\theta(Q_{1}/\alpha_{1})
\in
 \alpha_{1}^{4p+2}
S_{(\alpha_{1} )}. 
\end{align*}
Also
\begin{align*}
[
\Phi^{(1)}_{\zeta}
(\theta)]
 (\alpha_{2}) 
&=
Q_{1} (\nabla_{\theta} \zeta)(\alpha_{2})
-(4p+1)\theta(Q_{1}) \zeta(\alpha_{2})\\
&=
Q_{1} ({\theta} (\alpha_{2}^{4q+1}f_{2}))-
(4p+1)\theta(Q_{1}) 
\alpha_{2}^{4q+1}f_{2}\\
&=
Q_{1} \alpha_{2}^{4q+1}{\theta} (f_{2})
+
(4q+1) f_{2} \alpha_{2}^{4q}  Q_{1} \theta(\alpha_{2}) 
- 
{(4p+1)}  f_{2} \alpha_{2}^{4q+1} \theta(Q_{1})\\
&\in
 \alpha_{2}^{4q}
S_{(\alpha_{2} )}. 
\end{align*}
This shows $\Phi^{(1)}_{\zeta} (\theta) \in
D(\A, (4p+2, 4q))$.
Next we will prove
that 
$\Phi^{(1)}_{\zeta} (\partial_{x_{1} } ) $
and
$\Phi^{(1)}_{\zeta} (\partial_{x_{2} } ) $
  form an $S$-basis for $D(\A, (4p+2, 4q))$.
Define $M(\theta_{1}, \theta_{2})
:=
\left[
\theta_{i}(x_{j})
\right]_{1\leq i, j\leq 2}. 
$   
Then 
\begin{align*}
\det M(
\Phi^{(1)}_{\zeta} (\partial_{x_{1} }),
\Phi^{(1)}_{\zeta} (\partial_{x_{2} }))
&=
\det M(
Q_{1} \nabla_{\partial_{x_{1} } } \zeta,
Q_{1} \nabla_{\partial_{x_{2} } } \zeta
)\\ 
&~~~-(4p+1)
\det M(
Q_{1} \nabla_{\partial_{x_{1} } } \zeta,
(\partial_{x_{2} }Q_{1} ) \zeta)\\
&~~~-(4p+1)
\det M((\partial_{x_{1} }Q_{1})\zeta, 
Q_{1} \nabla_{\partial_{x_{2} } } \zeta).
\end{align*} 
Note
\[
x_{1} (\nabla_{\partial_{x_{1} }}\zeta)
+
x_{2} (\nabla_{\partial_{x_{2} }}\zeta)
=
\nabla_{E} \zeta
=
\left\{
1+h(p+q)
\right\}
\zeta
\]
{because 
$
\nabla_{\partial_{x_{1}} } \zeta,
\nabla_{\partial_{x_{2}} } \zeta
$ 
are a basis for $D(\A, (4p, 4q))$ and
${\rm pdeg}\, \zeta = 1+h(p+q)$. 
}
Thus
\begin{align*}
&~~~\det M(
\Phi^{(1)}_{\zeta} (\partial_{x_{1} }),
\Phi^{(1)}_{\zeta} (\partial_{x_{2} }))\\
&=
Q_{1}^{2}  \det M(
\nabla_{\partial_{x_{1} } } \zeta,
\nabla_{\partial_{x_{2} } } \zeta
) 
-\frac{(4p+1)Q_{1} x_{2} (\partial_{x_{2} }Q_{1} )}{1+h(p+q)}
\det M(
\nabla_{\partial_{x_{1} } } \zeta,
\nabla_{\partial_{x_{2} } } \zeta)\\
&~~~-\frac{(4p+1)Q_{1} x_{1} (\partial_{x_{1} }Q_{1} )}{1+h(p+q)}
\det 
M(
\nabla_{\partial_{x_{1} } } \zeta,
\nabla_{\partial_{x_{2} } } \zeta)\\
&=
\left\{
Q_{1}^{2} -
\frac{(4p+1)Q_{1} (x_{1} (\partial_{x_{1} }Q_{1}) +
x_{2} (\partial_{x_{2} }Q_{1}))}{1+h(p+q)}
 \right\}
\det
M(
\nabla_{\partial_{x_{1} } } \zeta,
\nabla_{\partial_{x_{2} } } \zeta)\\
&\doteq
\left\{
1 -
\frac{(4p+1)h}{2(1+h(p+q))}
 \right\}
Q_{1}^{2} Q_{1}^{4p} Q_{2}^{4q}
=
\frac{2-h(2p-2q+1)}{2(1+h(p+q))}
Q_{1}^{4p+2} Q_{2}^{4q}.
\end{align*} 
Note that 
$2-h(2p-2q+1)\neq 0$ 
and
$1+h(p+q)\neq 0$ because $h\geq 4$.
Therefore
$\Phi^{(1)}_{\zeta} (\partial_{x_{1} })$
and
$\Phi^{(1)}_{\zeta} (\partial_{x_{2} })$ 
 form an $S$-basis for
$D(\A, (4p+2, 4q))$
thanks to Theorem \ref{T. Abe's criterion 2 -theorem-}.   
Thus $\Phi^{(1)}_{\zeta}  $  
is an $S$-linear
{bijection.}
$\square$ 
\end{proof}

We may apply 
Proposition \ref{Phiisom} to get
the following proposition:

\begin{proposition}
\label{exceptionalbasis}
Suppose that $(a_{1}, a_{2})$ is exceptional
and that either $p \geq 0$ or
$q \geq 0$ in Table 2.
Then,
for $i=1, 2$, 
$
\Phi^{(i)}_{\zeta}(\theta_{1})$
and
$\Phi^{(i)}_{\zeta}(\theta_{2})$
form an $S$-basis for $D(\A, (a_{1} , a_{2} ))$
as in Table 2.
\end{proposition}

Proposition \ref{Estuniversal}
asserts that
$
E_{1}^{(s, t)}
$  
is $(s, t)$-universal 
when 
$s-t\in 2\Z, t\geq 0$
and
that
$
E_{2}^{(s, t)}
$  
is $(s, t)$-universal 
when 
$t-s\in 2\Z, s\geq 0$.
So Tables 1 and 2 show how to construct a basis for
$D(\A, (a_{1}, a_{2}  ))$ when $a_{1} \geq 0$ or
$a_{2} \geq 0$.
We will construct a basis for $D(\A, (a_{1}, a_{2}))$ 
in the remaining case that
$a_{1} < 0$ and $a_{2} < 0$. 
Let
\begin{align*} 
\Omega(\A, \bfk) 
&:=
(I^{*})^{-1} (D(\A, -\bfk))\\
&= 
\{
\omega\in\Omega(\A, -\infty)
\mid
I^{*} (\omega, d\alpha_{H})
\in 
\alpha_{H}^{-\bfk(H)} S_{(\alpha_{H})} 
\mbox{\rm ~for all~}
H\in \A
\}.
\end{align*}

\begin{theorem}
{\rm (Ziegler \cite{Ziegler}, Abe \cite[Theorem 1.7]{Abe})}
The natural $S$-bilinear coupling
\[
%
D(\A, \bfk) \times 
\Omega(\A, \bfk) 
\longrightarrow
S
\]
is non-degenerate and provides $S$-linear
isomorphisms:
\[
\alpha : D(\A, \bfk) \rightarrow \Omega(\A, \bfk)^{*},
\,\,\,\,
\beta : \Omega(\A, \bfk) \rightarrow D(\A, \bfk)^{*}.
\]
\end{theorem}

Thus we have the following proposition: 

\begin{proposition}
\label{bothnegative}
Let $(a_{1}, a_{2})\in (\Z_{<0})^{2}$
and $x_{1}, x_{2}  $ be an orthonormal basis.
Let $\theta_{1}, \theta_{2}$ be an $S$-basis for
$D(\A, (-a_{1}, -a_{2}))$.
Then 
\[
\eta_{1} 
:= 
g_{11} \partial_{x_{1} }  
+
g_{21} \partial_{x_{2} },  \,
\eta_{2} 
:= 
g_{12} \partial_{x_{1} }  
+
g_{22} \partial_{x_{2} },
\]
form an $S$-basis for $D(\A, (a_{1}, a_{2}))$.  
Here 
\[
\begin{pmatrix}
g_{11} & g_{12} \\
g_{21} & g_{22}
\end{pmatrix} 
=
\begin{pmatrix}
\theta_{1}(x_{1}) & \theta_{1}(x_{2}) \\
\theta_{2}(x_{1}) & \theta_{2}(x_{2})
\end{pmatrix}
^{-1}
=
Q_{1}^{a_{1} } Q_{2}^{a_{2} } 
\begin{pmatrix}
\theta_{2}(x_{2}) & -\theta_{1}(x_{2}) \\
-\theta_{2}(x_{1}) & \theta_{1}(x_{1})
\end{pmatrix}. 
\]

\end{proposition}

\section{Conclusion}

Let $\A$ 
be a two-dimensional irreducible Coxeter  arrangement 
such that
$|\A|$ is even with $|\A|\geq 4$.
We have constructed an explicit basis for $D(\A, (a_{1}, a_{2}))$ 
for an arbitrary equivariant multiplicity $\bfk = (a_{1}, a_{2})$
with $a_{1}, a_{2} \in \Z$.
Our recipes are presented in the Tables 1, 2,
Propositions \ref{ordinarybasis}, \ref{exceptionalbasis} 
and \ref{bothnegative}.   
Lastly we show Table 4
for the exponents.

\begin{table}[h]
\label{Table4}
\begin{center}
\begin{tabular}{|c|c|c|c|c|}
\hline
  $a_{1}$ 
& $a_{2}$ 
& $a_{1} - a_{2} $ 
& exponents of $(\A, (a_{1}, a_{2}))$ 
& their difference
\\
\hline
  odd 
& odd 
& $\equiv 0 \,(\mbox{\rm mod~} 4)$
& $
\frac{(a_{1} +a_{2} -2)h}{4} +1,
\frac{(a_{1} +a_{2} +2)h}{4} -1
 $ 
& $ h-2$\\
\hline
  odd 
& odd 
& $\equiv 2 \,(\mbox{\rm mod~} 4)$
& $
\frac{(a_{1} +a_{2})h}{4} +1,
\frac{(a_{1} +a_{2})h}{4} -1
 $ 
& $ 2$\\
\hline
  odd 
& even
& 
& 
$
\frac{(a_{1} +a_{2}-1)h}{4} +1,
\frac{(a_{1} +a_{2}+1)h}{4} -1
 $ 
& $ (h/2)-2$\\
\hline
  even 
& odd
& 
& 
$
\frac{(a_{1} +a_{2}-1)h}{4} +1,
\frac{(a_{1} +a_{2}+1)h}{4} -1
 $ 
& $ (h/2)-2$\\
\hline
even
& even
& 
& $
\frac{(a_{1} +a_{2})h}{4},
\frac{(a_{1} +a_{2})h}{4}
 $ 
& $0$\\
\hline
\end{tabular}
\end{center}
\caption{The exponents of $(\A, (a_{1}, a_{2}))
\,\,\,\,\,(a_{1}, a_{2} \in\Z )$} 
\end{table}


{{\bf Acknowledgement}
The author expresses his
gratitude to
Professor
Hiroaki Terao for his patient guidance
and many helpful discussions.
He also thanks 
the referee for proposing many
improvements of an earlier version.
}




{\footnotesize
{
\noindent
Atsushi WAKAMIKO \\
{\sc
2-8-11
Aihara, Midori-ku,
Sagamihara-shi, Kanagawa,
252-0141
Japan
\\
}
{\it e-mail}: atsushi.wakamiko@gmail.com
}
}

\end{document}